\documentclass[11pt]{amsart}
\usepackage{graphicx}
\usepackage{amsmath}
\usepackage{amssymb}
\usepackage{oldgerm}
\usepackage{mathdots}
\usepackage{stmaryrd}
\usepackage{bm}
\usepackage[all]{xy}
\usepackage{color}
\usepackage{enumerate}
\usepackage{fourier}
\usepackage{hyperref}
\usepackage{mathrsfs}
\usepackage{tikz}
\usetikzlibrary{arrows,%
  decorations.markings,%
  matrix,%
  shapes}

\numberwithin{equation}{section}

\theoremstyle{plain}

\numberwithin{equation}{section}

\newtheorem{theorem}[equation]{Theorem}
\newtheorem{corollary}[equation]{Corollary}
\newtheorem{proposition}[equation]{Proposition}
\newtheorem{lemma}[equation]{Lemma}

\theoremstyle{definition}
\newtheorem*{definition}{Definition}

\theoremstyle{definition}
\newtheorem{remark}[equation]{Remark}

\newtheorem*{conjecture}{Conjecture}

\newcommand{\A}{\mathscr{A}}
\newcommand{\B}{\mathscr{B}}
\newcommand{\C}{\mathscr{C}}
\newcommand{\T}{\mathscr{T}}
\newcommand{\X}{\mathscr{X}}
\newcommand{\Y}{\mathscr{Y}}

\newcommand\Ext{\operatorname{Ext}}
\renewcommand\mod{\operatorname{mod}}

\newcommand\Homol{\operatorname{H}}

\newcommand\op{\operatorname{op}}
\newcommand\ot{\otimes}
\newcommand\Hom{\operatorname{Hom}}

\newcommand\Hoch{\operatorname{HH}}

\newcommand\unit{\mathbf{1}}

\DeclareMathOperator{\Ker}{Ker}
\newcommand{\DOT}{\setlength{\unitlength}{1pt}\begin{picture}(2.5,2)(1,1)\put(2,3){\circle*{2}}\end{picture}}
\newcommand{\bu}{\DOT}

\newcommand{\Coh}{\operatorname{H}\nolimits}
\newcommand{\Ho}{\operatorname{\Coh^{\bu}}\nolimits}

\newcommand\Kdim{\operatorname{Kdim}}
\newcommand\lresdim{\operatorname{l-resdim}}
\newcommand\rresdim{\operatorname{r-resdim}}
\newcommand\add{\operatorname{add}}
\newcommand\repdim{\operatorname{repdim}}
\newcommand\gldim{\operatorname{gldim}}
\newcommand\perf{\operatorname{perf}}
\newcommand\derived{\operatorname{D^b}}
\newcommand\sing{\operatorname{D_{sg}}}
\newcommand{\stable}{\operatorname{stab}\nolimits}
\newcommand{\thick}{\operatorname{thick}\nolimits}

\newcommand{\az}{\mathfrak{a}}

\newcommand{\diagram}[3]{\matrix (#1) [matrix of math nodes,row
  sep={#2},column sep={#3},text height=1.5ex,text
  depth=0.25ex]}

\makeatletter
\def\blx@maxline{77}
\makeatother

\begin{document}

\title[Separable equivalences]{Separable equivalences, finitely generated cohomology and finite tensor categories}

\author{Petter Andreas Bergh}

\address{Petter Andreas Bergh \\ Institutt for matematiske fag \\
NTNU \\ N-7491 Trondheim \\ Norway} \email{petter.bergh@ntnu.no}


\begin{abstract}
We show that finitely generated cohomology is invariant under separable equivalences for all algebras. As a result, we obtain a proof of the finite generation of cohomology for finite symmetric tensor categories in characteristic zero, as conjectured by Etingof and Ostrik. Moreover, for such categories we also determine the representation dimension and the Rouquier dimension of the stable category. Finally, we recover a number of results on the cohomology of stably equivalent and singularly equivalent algebras.
\end{abstract}

\maketitle

\section{Introduction}\label{sec:intro}

Given an algebra over a commutative ring, the Hochschild cohomology ring acts on the cohomology of any pair of modules. As a result, one can define support varieties for modules in terms of the spectrum of the cohomology ring, introduced in \cite{SnashallSolberg}. As shown in \cite{EHSST}, when the cohomology ring is Noetherian, and the cohomology of the modules is finitely generated, then these support varieties encode important homological information. In this situation, the theory is very similar to the classical ones for finite groups, cocommutative Hopf algebras and commutative complete intersections. For example, the varieties detect the modules of finite projective dimension; these are precisely the modules with trivial varieties. 

In \cite{Kadison}, the notion of separably equivalent algebras was introduced. Several years later, it was shown in \cite{Linckelmann} that finitely generated cohomology is invariant under separable equivalences, provided the algebras involved are \emph{symmetric}. This was then used to study the cohomology of certain Hecke algebras of classical type.

In this paper, we prove that finitely generated cohomology is invariant under separable equivalences for \emph{all} algebras. We apply this first to skew group algebras, a class of algebras which by Deligne's theorem from \cite{Deligne} is closely linked to finite tensor categories. Namely, over an algebraically closed field of characteristic zero, every symmetric such category is equivalent to the representation category of a certain finite dimensional Hopf algebra which can be realized as the skew group algebra of an exterior algebra. As a result, we obtain a proof of the finite generation of cohomology for symmetric tensor categories in characteristic zero, as conjectured by Etingof and Ostrik in \cite{EtingofOstrik}. For such a category $\C$, we also determine the representation dimension, and the Rouquier dimension of the stable category; both these invariants are linked to the Krull dimension of the category, that is, the Krull dimension of the cohomology ring. Namely, we show that there are equalities
$$\Kdim \C = \dim \stable \C + 1 = \repdim \C -1 < \infty$$
where $\Kdim \C$ denotes the Krull dimension, $\stable \C$ denotes the stable category of $\C$, and $\repdim \C$ denotes the representation dimension. The latter was introduced for module categories by Auslander, and provides a measure of how far the category $\C$ is from having finite representation type. The Rouquier dimension of $\stable \C$ measures how many steps, or cones, one needs in order to generate the triangulated category from a single object.

In the final section, we explore some further applications of the result on finitely generated cohomology and separable equivalences. In particular, we recover a number of results on the cohomology of stably equivalent and singularly equivalent algebras.

\subsection*{Acknowledgments}
I would like to thank Markus Linckelmann, Cris Negron and Sarah Witherspoon for valuable comments and suggestions.

\section{A generalized Eckmann-Shapiro lemma}\label{sec:Eckmann-Shapiro}

The classical Eckmann-Shapiro lemma (cf.\ \cite[Corollary 2.8.4]{Benson}) provides isomorphisms of extension groups over different rings, via restriction and extension of scalars. Explicitly, if $R \to S$ is a homomorphism of rings, with $S$ projective as a right $R$-module, then for every $R$-module $M$, $S$-module $N$, and integer $n \ge 0$, there is an isomorphism
$$\Ext_R^n \left ( M,N \right ) \simeq \Ext_S^n \left ( S \otimes_R M, N \right )$$
of abelian groups. The proof applies the Hom-tensor adjunction, and the fact that the functors $S \otimes_R -$ and $\Hom_S(S,-)$ are exact (note that $\Hom_S(S,N)$ is precisely the restriction of $N$ to $R$).

\sloppy In this section, we record a generalized version of the Eckmann-Shapiro lemma, for abelian categories. The result is basically the first two parts of \cite[Proposition 9.1]{ErdmannSolbergWang}, but we include a proof because of mildly different assumptions and notation. Given an abelian category $\A$, two objects $M,N \in \A$, and an integer $n \ge 0$, we define $\Ext_{\A}^n(M,N)$ as the abelian group of equivalence classes of $n$-fold extensions of $M$ by $N$. When $\A$ has enough projective objects, this can be defined in terms of projective resolutions of $M$, like for modules over a ring. We denote the graded abelian group $\bigoplus_{n=0}^{\infty}\Ext_{\A}^n(M,N)$ by $\Ext_{\A}^*(M,N)$; using Yoneda products (that is, splicing of exact sequences), the graded group $\Ext_{\A}^*(M,M)$ becomes a graded ring, and $\Ext_{\A}^*(M,N)$ becomes a graded $\Ext_{\A}^*(N,N)$-$\Ext_{\A}^*(M,M)$-bimodule. We shall use the symbol $\circ$ to denote both the Yoneda product and composition of morphisms. Finally, an exact functor $F \colon \A \to \B$, from $\A$ to another abelian category $\B$, induces a homomorphism 
\begin{center}
\begin{tikzpicture}
\diagram{d}{3em}{3em}{
\Ext_{\A}^* \left ( M,M \right ) & \Ext_{\B}^* \left ( F(M),F(M) \right ) \\
 };
\path[->, font = \scriptsize, auto]
(d-1-1) edge node{$\varphi_M^F$} (d-1-2);
\end{tikzpicture}
\end{center}
of graded rings, in the obvious way.

\begin{theorem}\label{thm:adjointExt}
\sloppy Suppose that $\A$ and $\B$ are abelian categories with enough projective objects, and that $F \colon \A \to \B$ and $G \colon \B \to \A$ are exact functors forming an adjoint pair $(F,G)$. Then for all objects $X \in \A$ and $Y \in \B$, a choice of adjoint isomorphisms induces an isomorphism
\begin{center}
\begin{tikzpicture}
\diagram{d}{3em}{3em}{
\Ext_{\A}^* \left ( X,G(Y) \right ) & \Ext_{\B}^* \left ( F(X),Y \right ) \\
 };
\path[->, font = \scriptsize, auto]
(d-1-1) edge node{$\tau$} (d-1-2);
\end{tikzpicture}
\end{center}
of graded abelian groups. Furthermore, when we view $\Ext_{\A}^* \left ( X,G(Y) \right )$ and $\Ext_{\B}^* \left ( F(X),Y \right )$ as graded right $\Ext_{\A}^* \left ( X,X \right )$-modules -- the latter via the ring homomorphism 
\begin{center}
\begin{tikzpicture}
\diagram{d}{3em}{3em}{
\Ext_{\A}^* \left ( X,X \right ) & \Ext_{\B}^* \left ( F(X),F(X) \right ) \\
 };
\path[->, font = \scriptsize, auto]
(d-1-1) edge node{$\varphi_X^F$} (d-1-2);
\end{tikzpicture}
\end{center}
induced by $F$ -- then $\tau$ becomes an isomorphism of such. That is, if $\eta \in \Ext_{\A}^* \left ( X,G(Y) \right )$ and $\theta \in \Ext_{\A}^* \left ( X,X \right )$ are homogeneous elements, then $\tau ( \eta \circ \theta ) = \tau ( \eta ) \circ  \varphi_X^F ( \theta )$.
\end{theorem}

\begin{proof}
Let us start by fixing a natural bijection
\begin{center}
\begin{tikzpicture}
\diagram{d}{3em}{3em}{
\Hom_{\A} \left ( X,G(Y) \right ) & \Hom_{\B} \left ( F(X),Y \right ) \\
 };
\path[->, font = \scriptsize, auto]
(d-1-1) edge node{$\sigma_{X,Y}$} (d-1-2);
\end{tikzpicture}
\end{center}
for each object $X \in \A$ and $Y \in \B$; by \cite[Theorem IV.1.3]{MacLane2}, these bijections are isomorphisms of abelian groups. Note that in addition to being exact, the functor $F$ preserves projective objects, since it is left adjoint to an exact functor. Namely, if $P$ is a projective object in $\A$, then the functor $\Hom_{\B} \left ( F(P),- \right )$, which is naturally isomorphic to $\Hom_{\A} \left ( P, G(-) \right )$, is exact, since the latter functor is the composition of the two exact functors $G$ and $\Hom_{\A} \left ( P,- \right )$. Hence $F(P)$ is a projective object in $\B$.

Now fix objects $X \in \A$ and $Y \in \B$. Furthermore, fix a projective resolution $(P_*, d_*)$ of $X$ in $\A$; by the above, the sequence $(F(P_*),F(d_*))$ is then a projective resolution of $F(X)$ in $\B$. Applying $\Hom_{\A} (-,G(Y))$ to the former, and $\Hom_{\B} (-,Y )$ to the latter, we obtain a diagram
\begin{center}
\begin{tikzpicture}
\diagram{d}{3em}{3em}{
\cdots & {_{\A}( P_{n-1},G(Y) )} & {_{\A}(P_n,G(Y))} & {_{\A}(P_{n+1},G(Y))} & \cdots \\
\cdots & {_{\B}( F(P_{n-1}),Y )} & {_{\B}( F(P_n),Y )} & {_{\B}( F(P_{n+1}),Y )} & \cdots \\
 };
\path[->, font = \scriptsize, auto]
(d-1-1) edge (d-1-2)
(d-1-2) edge node{$d_n^*$} (d-1-3)
(d-1-3) edge node{$d_{n+1}^*$} (d-1-4)
(d-1-4) edge (d-1-5)
(d-2-1) edge (d-2-2)
(d-2-2) edge node{$F(d_n)^*$} (d-2-3)
(d-2-3) edge node{$F(d_{n+1})^*$} (d-2-4)
(d-2-4) edge (d-2-5)
(d-1-2) edge node{$\sigma_{P_{n-1},Y}$} (d-2-2)
(d-1-3) edge node{$\sigma_{P_n,Y}$} (d-2-3)
(d-1-4) edge node{$\sigma_{P_{n+1},Y}$} (d-2-4);
\end{tikzpicture}
\end{center}
which is commutative since the vertical adjoint isomorphisms are natural (in the diagram, we have abbreviated $\Hom_{\A} (P_i,G(Y))$ to ${_{\A}(P_i,G(Y))}$ in the upper row, and similarly in the lower row). The isomorphism of complexes now gives
\begin{eqnarray*}
\Ext_{\A}^n \left ( X,G(Y) \right ) & = & \Homol^n \left ( \Hom_{\A} (P_*,G(Y)) \right ) \\
& \simeq & \Homol^n \left ( \Hom_{\B} (F(P_*),Y) \right ) \\
& = & \Ext_{\B}^n \left ( F(X),Y \right )
\end{eqnarray*}
for each $n$, with the isomorphism induced by $\sigma_{P_n,Y}$. We thus obtain an isomorphism
\begin{center}
\begin{tikzpicture}
\diagram{d}{3em}{3em}{
\Ext_{\A}^* \left ( X, G(Y) \right ) & \Ext_{\B}^* \left ( F(X), Y\right ) \\
 };
\path[->, font = \scriptsize, auto]
(d-1-1) edge node{$\tau$} (d-1-2);
\end{tikzpicture}
\end{center}
of graded abelian groups. Explicitly, a homogeneous element $\eta \in \Ext_{\A}^* \left ( X, G(Y) \right )$ of degree $n$, represented by a map $f_\eta \colon P_n \to G(Y)$, is mapped to the degree $n$ element of $\Ext_{\B}^* \left ( F(X), Y\right )$ represented by the map $\sigma_{P_n,Y}(f_{\eta}) \colon F(P_n) \to Y$.

\sloppy It remains to show that $\tau$ is an isomorphism of graded right $\Ext_{\A}^* \left ( X,X \right )$-modules. Take the element $\eta \in \Ext_{\A}^n \left ( X,G(Y) \right )$ above, and an element $\theta \in \Ext_{\A}^m \left ( X,X \right )$ for some $m \ge 0$. Furthermore, represent the latter by a map $g_{\theta} \colon P_m \to X$. Lifting $g_{\theta}$ along $(P_*,d_*)$, we obtain maps $g_i \colon P_{m+i} \to P_i$, and a commutative diagram
\begin{center}
\begin{tikzpicture}
\diagram{d}{3em}{3em}{
\cdots & P_{m+n+1} & P_{m+n} & P_{m+n-1} & \cdots \\
\cdots & P_{n+1} & P_n & P_{n-1} & \cdots \\
&& G(Y) \\
 };
\path[->, font = \scriptsize, auto]
(d-1-1) edge (d-1-2)
(d-1-2) edge node{$d_{m+n+1}$} (d-1-3)
(d-1-2) edge node{$g_{n+1}$} (d-2-2)
(d-1-3) edge node{$d_{m+n}$} (d-1-4)
(d-1-3) edge node{$g_n$} (d-2-3)
(d-1-4) edge (d-1-5)
(d-1-4) edge node{$g_{n-1}$} (d-2-4)
(d-2-1) edge (d-2-2)
(d-2-2) edge node{$d_{n+1}$} (d-2-3)
(d-2-3) edge node{$d_n$} (d-2-4)
(d-2-3) edge node{$f_{\eta}$} (d-3-3)
(d-2-4) edge (d-2-5);
\end{tikzpicture}
\end{center}
The map $f_{\eta} \circ g_n$ represents the element $\eta \circ \theta$, and so $\sigma_{P_{m+n},Y}( f_{\eta} \circ g_n )$ represents $\tau ( \eta \circ \theta )$. Now $F( g_{\theta} )$ represents the element $\varphi_X^F ( \theta )$, and we obtain a lifting of this map along the projective resolution $( F(P_*),F(d_*) )$ by applying $F$ to the $g_i$. Hence the map $\sigma_{P_n,Y}( f_{\eta} ) \circ F(g_n)$ represents the element $\tau ( \eta ) \circ \varphi_X^F ( \theta )$. But the naturality of $\sigma_{P_n,Y}$, applied to the map $g_n \colon P_{m+n} \to P_n$, gives a commutative diagram
\begin{center}
\begin{tikzpicture}
\diagram{d}{3em}{4em}{
\Hom_{\A} \left ( P_n, G(Y) \right ) & \Hom_{\B} \left ( F(P_n), Y \right ) \\
\Hom_{\A} \left ( P_{m+n}, G(Y) \right ) & \Hom_{\B} \left ( F(P_{m+n}), Y \right ) \\
 };
\path[->, font = \scriptsize, auto]
(d-1-1) edge node{$\sigma_{P_n,Y}$} (d-1-2)
(d-2-1) edge node{$\sigma_{P_{m+n},Y}$} (d-2-2)
(d-1-1) edge node{$g_n^*$} (d-2-1)
(d-1-2) edge node{$F(g_n)^*$} (d-2-2);
\end{tikzpicture}
\end{center}
Tracing $f_{\eta}$, we see that $\sigma_{P_{m+n},Y}( f_{\eta} \circ g_n ) = \sigma_{P_n,Y}( f_{\eta} ) \circ F(g_n)$, hence $\tau ( \eta \circ \theta ) = \tau ( \eta ) \circ \varphi_X^F ( \theta )$. This shows that $\tau$ is an isomorphism of graded right $\Ext_{\A}^* \left ( X,X \right )$-modules.
\end{proof}

\section{Separably equivalent algebras}\label{sec:separably}

In this section, we apply Theorem \ref{thm:adjointExt} to bimodules over algebras, and show that finite generation of cohomology transfers between algebras that are linked in a certain way. Let us fix a commutative Noetherian ring $k$, together with two Noetherian $k$-algebras $A$ and $B$; thus there exist ring homomorphisms from $k$ to the centers of $A$ and $B$, through which these rings are finitely generated as $k$-modules. We make the further assumption that both $A$ and $B$ are projective as $k$-modules. Finally, all modules are assumed to be finitely generated left modules, unless otherwise specified. We denote by $\mod A$ the category of finitely generated left $A$-modules.

\begin{definition}\cite{BerghErdmann, Kadison, Linckelmann}
The algebra $B$ \emph{separably divides} the algebra $A$ if there exist bimodules ${_AU_B}$ and ${_BV_A}$ -- projective on both sides -- with the property that $B$ is a direct summand of $V \ot_A U$ as a $B$-bimodule. If there exists such a pair of bimodules such that in addition $A$ is a direct summand of $U \ot_B V$ as an $A$-bimodule, then $A$ and $B$ are \emph{separably equivalent}.
\end{definition}

\begin{remark}
Of course, if $A$ and $B$ are separably equivalent, then in particular $B$ separably divides $A$, and $A$ separably divides $B$. But the converse is also true. For suppose that $B$ separably divides $A$ through bimodules ${_AU_B}$ and ${_BV_A}$, and that $A$ separably divides $B$ through bimodules ${_AU'_B}$ and ${_BV'_A}$. Then by taking $X = U \oplus U'$ and $Y = V \oplus V'$, we see that $B$ is a direct summand of $Y \otimes_A X$ as a $B$-bimodule, and that $A$ is a direct summand of $X \otimes_B Y$ as an $A$-bimodule. This simple fact escaped the present author in \cite{BerghErdmann}; see the top of page 2507.
\end{remark}

Denote the enveloping algebra $A \ot_k A^{\op}$ of $A$ by $A^e$; left modules over this algebra are the same as bimodules over $A$. Furthermore, denote by $\Hoch^*(A)$ the Hochschild cohomology ring $\Ext_{A^e}^*(A,A)$; since $A$ is projective as a $k$-module, this definition of the Hochschild cohomology ring agrees with the original definition of Hochschild, up to isomorphism. By a classical result of Gerstenhaber, this is a graded-commutative ring. Now for every $A$-module $M$, the tensor product $- \ot_A M$ induces a homomorphism
\begin{center}
\begin{tikzpicture}
\diagram{d}{3em}{3em}{
\Hoch^*(A) & \Ext_A^*(M,M) \\
 };
\path[->, font = \scriptsize, auto]
(d-1-1) edge node{$\varphi_M$} (d-1-2);
\end{tikzpicture}
\end{center}
of graded $k$-algebras. Thus if $N$ is another $A$-module, then $\Ext_A^*(M,N)$ becomes a left $\Hoch^*(A)$-module via $\varphi_N$ and the Yoneda product, and a right $\Hoch^*(A)$-module via $\varphi_M$ and the Yoneda product. However, by \cite[Corollary 1.3]{SnashallSolberg}, the left and the right module actions coincide up to a sign, for homogeneous elements. That is, if $\eta \in \Hoch^m(A)$ and $\theta \in \Ext_A^n(M,N)$, then
$$\varphi_N ( \eta ) \circ \theta = (-1)^{nm} \theta \circ \varphi_M ( \eta )$$
as an element of $\Ext_A^{m+n}(M,N)$.

\begin{definition}
(1) The algebra $A$ satisfies \textbf{Fg} if the following hold: the Hochschild cohomology ring $\Hoch^*(A)$ is Noetherian, and $\Ext_A^*(M,M)$ is a finitely generated $\Hoch^*(A)$-module for every $A$-module $M$.

(2) The algebra $A$ satisfies \textbf{Fgb} if the following hold: the Hochschild cohomology ring $\Hoch^*(A)$ is Noetherian, and $\Ext_{A^e}^*(A,X)$ is a finitely generated right $\Hoch^*(A)$-module for every $A$-bimodule $X$ (where the module structure is defined via the Yoneda product).
\end{definition}

Note that for the assumption \textbf{Fg}, the requirement that $\Ext_A^*(M,M)$ is a finitely generated $\Hoch^*(A)$-module for every $A$-module $M$ is equivalent to requiring that $\Ext_A^*(M,N)$ is finitely generated over $\Hoch^*(A)$ for all pairs of $A$-modules $M,N$. This follows from the simple fact that $\Ext_A^*(M,N)$ is a direct summand of $\Ext_A^*(M \oplus N,M \oplus N)$ as a module over $\Hoch^*(A)$. This finite generation assumption was central in \cite{EHSST}, where it was used to explore homological properties for the support varieties that one can attach to $A$-modules, using the maximal ideal spectrum of the Hochschild cohomology ring. For the assumption \textbf{Fgb}, the letter <<b>> indicates that we are looking at bimodules. In general, there does not seem to be a connection between the two finiteness assumptions, unless we impose some restrictions on the ground ring $k$.

\begin{lemma}\label{lem:finitenessassumptions}
Let $k$ be a commutative Noetherian ring, and $A$ a Noetherian $k$-algebra which is projective as a $k$-module.

\emph{(1)} If $k$ is semisimple and $A$ satisfies \emph{\textbf{Fgb}}, then it also satisfies \emph{\textbf{Fg}}.

\emph{(2)} Suppose that $k$ is a field and that $A / \mathfrak{r} \otimes_k A / \mathfrak{r}$ is semisimple, where $\mathfrak{r}$ is the radical of $A$. Then $A$ satisfies \emph{\textbf{Fgb}} if and only if it satisfies \emph{\textbf{Fg}}.
\end{lemma}

\begin{proof}
When the ring $k$ is semisimple, then by \cite[Corollary IX.4.4]{CartanEilenberg} there is an isomorphism $\Ext_{A^e}^*(A, \Hom_k(M,N)) \simeq \Ext_A^*(M,N)$ for all $A$-modules $M$ and $N$. Thus if \textbf{Fgb} holds, then so does \textbf{Fg}. For the second part, note that by \cite[Lemma 7.6]{PsaroudakisSkartsaeterhagenSolberg}, the assumption implies that all the simple $A$-bimodules appear as summands of $A / \mathfrak{r} \otimes_k A / \mathfrak{r}$. It follows from this that every simple $A$-bimodule is of the form $\Hom_k(S,T)$, where $S$ and $T$ are simple $A$-modules. The argument from the proof of \cite[Proposition 2.4]{EHSST} now carries over.
\end{proof}

\begin{remark}\label{rem:conditionssatisfied}
(1) When $k$ is an algebraically closed field, then $A / \mathfrak{r} \otimes_k A / \mathfrak{r}$ is automatically semisimple, by the classical theorem of Wedderburn-Artin. The same conclusion holds when $k$ is a perfect field, by \cite[Lemma 5.3.8 and Corollary 5.3.10]{Zimmermann}. However, there are many other settings in which the same holds, as pointed out in \cite[paragraph following Example 7.7]{PsaroudakisSkartsaeterhagenSolberg}. For example, it holds when $A$ is the quotient of a path algebra by an admissible ideal, regardless of the ground field $k$. It also holds when $A / \mathfrak{r}$ is a separable $k$-algebra.

(2) Since the ring $\Hoch^*(A)$ is graded-commutative, it is Noetherian if and only if it is right Noetherian. Hence the assumption \textbf{Fgb} is equivalent to the following: the right $\Hoch^*(A)$-module $\Ext_{A^e}^*(A,X)$ is Noetherian for every $A$-bimodule $X$.

(3) When discussing the assumption \textbf{Fgb}, then given an $A$-bimodule $X$, we have been viewing $\Ext_{A^e}^*(A,X)$ as a right $\Hoch^*(A)$-module, using the Yoneda product. However, as for one-sided modules, it is also a left $\Hoch^*(A)$-module, by using the tensor product $- \ot_A X$ followed by the Yoneda product. By \cite[Theorem 1.1]{SnashallSolberg}, the left and the right actions coincide up to a sign, for homogeneous elements.
\end{remark}

We shall prove that if the algebra $A$ satisfies \textbf{Fgb}, and the algebra $B$ separably divides $A$, then $B$ also satisfies \textbf{Fgb}. In particular, if $A$ and $B$ are separably equivalent, then $A$ satisfies \textbf{Fgb} if and only if $B$ does. This was proved in \cite{Linckelmann} for symmetric separably equivalent algebras, but, as we shall see, it holds in full generality. The proof is modelled on that of \cite{Linckelmann}, but with appropriate modifications throughout. 

We start with an elementary lemma on isomorphisms of certain bimodules involving $\Hom$s and tensor products. Note first that if ${_AU_B}$ and ${_AW_B}$ are bimodules, then $\Hom_{B^{\op}}(U,B)$ is a $B$-$A$-bimodule, and $\Hom_{B^{\op}}(U,W)$ is an $A$-bimodule; the former by $(b \cdot h \cdot a)(u) = bh(au)$, and the latter by $(a \cdot g \cdot a')(u) = ag(a'u)$. Similarly, if ${_BV_A}$ and ${_BY_B}$ are bimodules, then $\Hom_B(V,Y)$ is an $A$-$B$-bimodule by $(a \cdot f \cdot b)(v) = f(va)b$. Finally, we denote the $B$-$A$-bimodule $\Hom_{B^{\op}}(U,B)$ by $U^*$, and the $A$-$B$-bimodule $\Hom_B(V,B)$ by ${^*V}$.

\begin{lemma}\label{lem:hom-tensor-iso}
If ${_AU_B}$ and ${_BV_A}$ are bimodules, both projective as $B$-modules, then the following hold.

\emph{(1)} For every bimodule ${_AW_B}$, there is an isomorphism
\begin{center}
\begin{tikzpicture}
\diagram{d}{3em}{3em}{
W \otimes_B U^* & \Hom_{B^{\op}}(U,W) \\
 };
\path[->, font = \scriptsize, auto]
(d-1-1) edge node{$\psi$} (d-1-2);
\end{tikzpicture}
\end{center}
of $A$-bimodules given by the linear extension of $w \otimes h \mapsto m_{w,h}$ for $w \in W$ and $h \in U^*$, with $m_{w,h}(u) = w \cdot h(u)$ for $u \in U$. Furthermore, this isomorphism is natural in $W$.

\emph{(2)} For every bimodule ${_BY_B}$, there is an isomorphism
\begin{center}
\begin{tikzpicture}
\diagram{d}{3em}{3em}{
{^*V} \otimes_B Y & \Hom_B(V,Y) \\
 };
\path[->, font = \scriptsize, auto]
(d-1-1) edge node{$\rho$} (d-1-2);
\end{tikzpicture}
\end{center}
of $A$-$B$-bimodules given by the linear extension of $\alpha \otimes y \mapsto w_{\alpha,y}$ for $\alpha \in {^*V}$ and $y \in Y$, with $w_{\alpha,y}(v) = \alpha (v) \cdot y$ for $v \in V$. Furthermore, this isomorphism is natural in $Y$.
\end{lemma}

\begin{proof}
We only prove (1); the proof of (2) is similar. Note first that given $w$ and $h$ as in the statement, the map $m_{w,h}$ really is a homomorphism of $B^{\op}$-modules, since $h$ is. Furthermore, the map $\psi$ is well defined; given $b \in B$, it is straightforward to check that $m_{w \cdot b,h} = m_{w,b \cdot h}$, hence the map 
\begin{center}
\begin{tikzpicture}
\diagram{d}{3em}{3em}{
W \times U^* & \Hom_{B^{\op}}(U,W) \\
 };
\path[->, font = \scriptsize, auto]
(d-1-1) edge (d-1-2);
\end{tikzpicture}
\end{center}
given by $(w,h) \mapsto m_{w,h}$ is $B$-balanced.

To show that $\psi$ is an isomorphism of abelian groups, we disregard the left $A$-module structure of $U$ (and $W$). Replacing $U$ with the right $B$-module $B_B$, we obtain a commutative diagram
\begin{center}
\begin{tikzpicture}
\diagram{d}{3em}{3em}{
W \otimes_B \Hom_{B^{\op}}(B_B,B) & \Hom_{B^{\op}}(B_B,W) \\
W \otimes_B B & W \\
 };
\path[->, font = \scriptsize, auto]
(d-1-1) edge node{$\psi$} (d-1-2)
(d-1-1) edge (d-2-1)
(d-1-2) edge (d-2-2)
(d-2-1) edge (d-2-2);
\end{tikzpicture}
\end{center}
in which the vertical maps and the bottom map are the canonical isomorphisms. Thus $\psi$ is an isomorphism in this case, and the argument extends by first replacing $B_B$ with a free right $B$-module, and finally with a projective right $B$-module. This shows that $\psi$ is an isomorphism of abelian groups.

Now take any $w \in W, h \in U^*$ and $a,a' \in A$. By definition, for every $u \in U$ we obtain
\begin{eqnarray*}
m_{a \cdot w, h \cdot a'}(u) & = & (aw) \cdot \left ( (h \cdot a')(u) \right ) \\
& = & (aw) \cdot h(a'u) \\
& = & a \left ( w \cdot h(a'u) \right ) \\
& = & a m_{w,h}(a'u) \\
& = & \left ( a \cdot m_{w,h} \cdot a' \right ) (u)
\end{eqnarray*}
hence $m_{a \cdot w, h \cdot a'} = a \cdot m_{w,h} \cdot a'$. Consequently, the isomorphism $\psi$ is one of $A$-bimodules. Finally,  a straightforward verification shows that it is also natural in $W$; for a homomorphism $W \to W'$ of $A$-$B$-bimodules, one uses its $B$-linearity to see this. 
\end{proof}

The next result establishes an adjoint pair of functors between $A$-bimodules and $B$-bimodules.

\begin{proposition}\label{prop:adjoint}
\sloppy Suppose that ${_AU_B}$ and ${_BV_A}$ are bimodules which are projective as $B$-modules, and consider the functors 
\begin{center}
\begin{tikzpicture}
\diagram{d}{3em}{5em}{
\mod A^e & \mod B^e \\
 };
\path[->, font = \scriptsize, auto]
(d-1-1) edge[bend left] node{$V \otimes_A - \otimes_A U$} (d-1-2)
(d-1-2) edge[bend left] node{${^*V} \otimes_B - \otimes_B  U^*$} (d-1-1);
\end{tikzpicture}
\end{center}
Then $\left ( V \otimes_A - \otimes_A U, {^*V} \otimes_B - \otimes_B  U^* \right )$ is an adjoint pair; given bimodules ${_AX_A}$ and ${_BY_B}$, there is an isomorphism
\begin{center}
\begin{tikzpicture}
\diagram{d}{3em}{3em}{
\Hom_{A^e} \left ( X, {^*V} \otimes_B Y \otimes_B  U^* \right ) & \Hom_{B^e} \left ( V \otimes_A X \otimes_A U, Y \right ) \\
 };
\path[->, font = \scriptsize, auto]
(d-1-1) edge node{$\sigma_{X,Y}$} (d-1-2);
\end{tikzpicture}
\end{center}
of abelian groups, natural in both $X$ and $Y$.
\end{proposition}

\begin{proof}
We divide the proof into four steps, in each of these producing a natural isomorphism. The composition of these isomorphisms is then the map we seek.

(1) With the $A$-$B$-bimodule ${^*V} \otimes_B Y$ as $W$ in Lemma \ref{lem:hom-tensor-iso}(1), we obtain an isomorphism
\begin{center}
\begin{tikzpicture}
\diagram{d}{3em}{3em}{
{^*V} \otimes_B Y \otimes_B U^* & \Hom_{B^{\op}}(U,{^*V} \otimes_B Y) \\
 };
\path[->, font = \scriptsize, auto]
(d-1-1) edge node{$\psi$} (d-1-2);
\end{tikzpicture}
\end{center}
of $A$-bimodules, natural in $Y$. Applying $\Hom_{A^e}(X,-)$ to this isomorphism, we obtain the isomorphism
\begin{center}
\begin{tikzpicture}
\diagram{d}{3em}{3em}{
\Hom_{A^e} \left ( X, {^*V} \otimes_B Y \otimes_B U^* \right ) & \Hom_{A^e} \left ( X, \Hom_{B^{\op}}(U,{^*V} \otimes_B Y) \right ) \\
 };
\path[->, font = \scriptsize, auto]
(d-1-1) edge node{$\psi_*$} (d-1-2);
\end{tikzpicture}
\end{center}
Since $\psi$ is natural in $Y$, so is $\psi_*$. Moreover, the latter is trivially natural in $X$.

(2) Let ${_AW_B}$ be a bimodule, and consider $\Hom_{A^e} \left ( X, \Hom_{B^{\op}}(U,W) \right )$. The bimodule version of the Hom-tensor adjunction (cf.\ \cite[Corollary V.3.2]{MacLane1}) gives a natural isomorphism
\begin{center}
\begin{tikzpicture}
\diagram{d}{3em}{3em}{
\Hom_{A^e} \left ( X, \Hom_{B^{\op}}(U,W) \right ) & \Hom_{A \otimes_k B^{\op}} \left ( X \otimes_A U, W \right )  \\
 };
\path[->, font = \scriptsize, auto]
(d-1-1) edge (d-1-2);
\end{tikzpicture}
\end{center}
In particular, with $W = {^*V} \otimes_B Y$, we obtain an isomorphism
\begin{center}
\begin{tikzpicture}
\diagram{d}{3em}{3em}{
\Hom_{A^e} \left ( X, \Hom_{B^{\op}}(U,{^*V} \otimes_B Y) \right ) & \Hom_{A \otimes_k B^{\op}} \left ( X \otimes_A U, {^*V} \otimes_B Y \right )  \\
 };
\path[->, font = \scriptsize, auto]
(d-1-1) edge node{$\lambda$} (d-1-2);
\end{tikzpicture}
\end{center}
which is natural in $X$ and $Y$.

(3) This step is similar to (1). By Lemma \ref{lem:hom-tensor-iso}(2), there is an isomorphism
\begin{center}
\begin{tikzpicture}
\diagram{d}{3em}{3em}{
{^*V} \otimes_B Y & \Hom_B(V,Y) \\
 };
\path[->, font = \scriptsize, auto]
(d-1-1) edge node{$\rho$} (d-1-2);
\end{tikzpicture}
\end{center}
of $A$-$B$-bimodules, natural in $Y$. Applying $\Hom_{A \otimes_k B^{\op}} \left ( X \otimes_A U,- \right )$ to $\rho$, we obtain the isomorphism
\begin{center}
\begin{tikzpicture}
\diagram{d}{3em}{3em}{
\Hom_{A \otimes_k B^{\op}} \left ( X \otimes_A U, {^*V} \otimes_B Y \right ) & \Hom_{A \otimes_k B^{\op}} \left ( X \otimes_A U, \Hom_B(V,Y) \right ) \\
 };
\path[->, font = \scriptsize, auto]
(d-1-1) edge node{$\rho_*$} (d-1-2);
\end{tikzpicture}
\end{center}
which is also natural in $Y$. This isomorphism is trivially natural in $X$.

(4) This last step is similar to (2). Let ${_AW_B}$ be a bimodule, and consider $\Hom_{A \otimes_k B^{\op}} \left ( W, \Hom_B(V,Y) \right )$. Again, the (appropriate version of the) bimodule version of the Hom-tensor adjunction gives a natural isomorphism
\begin{center}
\begin{tikzpicture}
\diagram{d}{3em}{3em}{
\Hom_{A \otimes_k B^{\op}} \left ( W, \Hom_B(V,Y) \right ) & \Hom_{B^e} \left ( V \otimes_A W, Y \right ) \\
 };
\path[->, font = \scriptsize, auto]
(d-1-1) edge (d-1-2);
\end{tikzpicture}
\end{center}
Then with $W = X \otimes_A U$ we obtain an isomorphism
\begin{center}
\begin{tikzpicture}
\diagram{d}{3em}{3em}{
\Hom_{A \otimes_k B^{\op}} \left ( X \otimes_A U, \Hom_B(V,Y) \right ) & \Hom_{B^e} \left ( V \otimes_A X \otimes_A U, Y \right ) \\
 };
\path[->, font = \scriptsize, auto]
(d-1-1) edge node{$\mu$} (d-1-2);
\end{tikzpicture}
\end{center}
which is natural in $X$ and $Y$.

Finally, the composition $\mu \circ \rho_* \circ \lambda \circ \psi_*$ of the isomorphisms from the above four steps gives an isomorphism
\begin{center}
\begin{tikzpicture}
\diagram{d}{3em}{3em}{
\Hom_{A^e} \left ( X, {^*V} \otimes_B Y \otimes_B  U^* \right ) & \Hom_{B^e} \left ( V \otimes_A X \otimes_A U, Y \right ) \\
 };
\path[->, font = \scriptsize, auto]
(d-1-1) edge (d-1-2);
\end{tikzpicture}
\end{center}
which is also natural in $X$ and $Y$.
\end{proof}

When the bimodules $U$ and $V$ from the previous result are also projective as $A$-modules, then the two functors in the adjoint pair are exact. We can then apply Theorem \ref{thm:adjointExt}, and obtain the following.

\begin{corollary}\label{cor:ext-iso}
If ${_AU_B}, {_BV_A}, {_AX_A}$ and ${_BY_B}$ are bimodules, with $U$ and $V$ projective as one-sided modules, then the following hold.

\emph{(1)} The functor $V \otimes_A - \otimes_A U$, induces a homomorphism
\begin{center}
\begin{tikzpicture}
\diagram{d}{3em}{3em}{
\Ext_{A^e}^* \left ( X,X \right ) & \Ext_{B^e}^* \left ( V \otimes_A X \otimes_A U, V \otimes_A X \otimes_A U \right ) \\
 };
\path[->, font = \scriptsize, auto]
(d-1-1) edge node{$\varphi_X^{V,U}$} (d-1-2);
\end{tikzpicture}
\end{center}
of graded $k$-algebras.

\emph{(2)} There is an isomorphism
\begin{center}
\begin{tikzpicture}
\diagram{d}{3em}{3em}{
\Ext_{A^e}^* \left ( X, {^*V} \otimes_B Y \otimes_B  U^* \right ) & \Ext_{B^e}^* \left ( V \otimes_A X \otimes_A U, Y\right ) \\
 };
\path[->, font = \scriptsize, auto]
(d-1-1) edge node{$\tau$} (d-1-2);
\end{tikzpicture}
\end{center}
of graded $k$-modules. Furthermore, when we view $\Ext_{A^e}^* \left ( X, {^*V} \otimes_B Y \otimes_B  U^* \right )$ and $\Ext_{B^e}^* \left ( V \otimes_A X \otimes_A U, Y\right )$ as right $\Ext_{A^e}^* \left ( X,X \right ) $-modules -- the latter via the ring homomorphism $\varphi_X^{V,U}$ from \emph{(1)} -- then $\tau$ becomes an isomorphism of such.
\end{corollary}

\begin{proof}
The proof of (1) just follows from the fact that the bimodules $U$ and $V$ are projective as $A$-modules, so that the functor $V \otimes_A - \otimes_A U$ is exact. The proof of (2) is a direct application of Theorem \ref{thm:adjointExt} and Proposition \ref{prop:adjoint}, and the fact that the functor ${^*V} \otimes_B - \otimes_B  U^*$ is also exact. Namely, by \cite[Theorem V.4.1]{MacLane1}, the bimodules ${^*V}$ and $U^*$ are projective as $B$-modules.
\end{proof}

We can now prove the main result in this section.

\begin{theorem}\label{thm:FGBsepequiv}
Let $k$ be a commutative Noetherian ring, and $A$ and $B$ two Noetherian $k$-algebras which are projective as $k$-modules. Furthermore, suppose that the algebra $B$ separably divides the algebra $A$, and that $A$ satisfies \emph{\textbf{Fgb}}. Then $B$ also satisfies \emph{\textbf{Fgb}}. In particular, if $A$ and $B$ are separably equivalent, then $A$ satisfies \emph{\textbf{Fgb}} if and only if $B$ does.
\end{theorem}

\begin{proof}
\sloppy If $B$ separably divides $A$, then by definition there exist bimodules ${_AU_B}$ and ${_BV_A}$ which are projective on both sides, and such that $B$ is a direct summand of $V \ot_A U$ as a $B$-bimodule. Take any $B$-bimodule $Y$, and set $X = A$ in Corollary \ref{cor:ext-iso}. We then obtain a ring homomorphism
\begin{center}
\begin{tikzpicture}
\diagram{d}{3em}{3em}{
\Hoch^*(A) & \Ext_{B^e}^* \left ( V \otimes_A U, V \otimes_A U \right ) \\
 };
\path[->, font = \scriptsize, auto]
(d-1-1) edge node{$\varphi_A^{V,U}$} (d-1-2);
\end{tikzpicture}
\end{center}
of graded $k$-algebras, and an isomorphism
\begin{center}
\begin{tikzpicture}
\diagram{d}{3em}{3em}{
\Ext_{A^e}^* \left ( A, {^*V} \otimes_B Y \otimes_B  U^* \right ) & \Ext_{B^e}^* \left ( V \otimes_A U, Y \right ) \\
 };
\path[->, font = \scriptsize, auto]
(d-1-1) edge node{$\tau$} (d-1-2);
\end{tikzpicture}
\end{center}
of graded right $\Hoch^*(A)$-modules, where the $\Hoch^*(A)$-module structure on $\Ext_{B^e}^* \left ( V \otimes_A U, Y \right )$ is via $\varphi_A^{V,U}$.

By assumption and Remark \ref{rem:conditionssatisfied}(2), given any $A$-bimodule $X$, the right $\Hoch^*(A)$-module $\Ext_{A^e}^* \left ( A, X \right )$ is Noetherian. In particular, this holds for the $\Hoch^*(A)$-module $\Ext_{A^e}^* \left ( A, {^*V} \otimes_B Y \otimes_B  U^* \right )$. But then the same must hold for $\Ext_{B^e}^* \left ( V \otimes_A U, Y \right )$, since the isomorphism $\tau$ is one of right $\Hoch^*(A)$-modules. However, since $\Ext_{B^e}^* \left ( V \otimes_A U, Y \right )$ is an $\Hoch^*(A)$-module via the ring homomorphism $\varphi_A^{V,U}$, we conclude that $\Ext_{B^e}^* \left ( V \otimes_A U, Y \right )$ is a Noetherian right module over $\Ext_{B^e}^* \left ( V \otimes_A U, V \otimes_A U \right )$, for every $B$-bimodule $Y$. Then since $B$ is a direct summand of $V \otimes_A U$ as a $B$-bimodule, the right $\Hoch^*(B)$-module $\Ext_{B^e}^* \left ( B, Y \right )$ is Noetherian for every $B$-bimodule $Y$ (cf.\ \cite[Lemma 4.3]{Linckelmann}). Consequently, by Remark \ref{rem:conditionssatisfied}(2) again, the algebra $B$ satisfies \textbf{Fgb}.
\end{proof}

Using Lemma \ref{lem:finitenessassumptions} and Remark \ref{rem:conditionssatisfied}, we immediately obtain the following.

\begin{corollary}\label{cor:FGsepequiv}
Let $k$ be a field, and $A$ and $B$ two finite dimensional $k$-algebras with $A / \mathfrak{r}_A \otimes_k A / \mathfrak{r}_A$ and 
$B / \mathfrak{r}_B \otimes_k B / \mathfrak{r}_B$ semisimple, where $\mathfrak{r}_A$ is the radical of $A$, and $\mathfrak{r}_B$ is the radical of $B$ (as happens for example when $k$ is algebraically closed or perfect). Furthermore, suppose that $B$ separably divides $A$, and that $A$ satisfies \emph{\textbf{Fg}}. Then $B$ also satisfies \emph{\textbf{Fg}}. In particular, if $A$ and $B$ are separably equivalent, then $A$ satisfies \emph{\textbf{Fg}} if and only if $B$ does.
\end{corollary}

\section{Skew group algebras and finite tensor categories}\label{sec:applications1}

An important class of algebras for which Theorem \ref{thm:FGBsepequiv} applies is formed by the skew group algebras. These algebras are particularly ubiquitous in the theory of Hopf algebras, and they are intimately linked to finite symmetric tensor categories in characteristic zero, via Deligne's famous theorem.

As in the previous section, let $k$ be a commutative Noetherian ring, and $A$ a Noetherian $k$-algebra which is projective as a $k$-module. All modules are assumed to be finitely generated left modules, unless otherwise specified. Furthermore, suppose that $G$ is a finite group acting on $A$, in terms of a homomorphism from $G$ to the multiplicative group of $k$-algebra automorphisms of $A$. The associated \emph{skew group algebra} $A \rtimes G$ is the $k$-algebra whose underlying $k$-module is the tensor product $A \otimes_k kG$, and with multiplication defined by
$$( a \otimes g )( b \otimes h ) = a({^gb}) \otimes gh$$
for $a,b \in A$ and $g,h \in G$. This algebra is also called the \emph{smash product} of $A$ and $kG$, and then often denoted by $A \# kG$. 

As a left $A$-module, the skew group algebra decomposes into a direct sum 
$$A \rtimes G = \bigoplus_{g \in G} A \otimes g$$
This is in particular a decomposition of $k$-modules, and then with each summand isomorphic to $A$. Therefore, since the group $G$ is finite, the skew group algebra is again a Noetherian $k$-algebra, and it is projective as a $k$-module since $A$ is.

\begin{theorem}\label{thm:skewgroupalgebras}
Let $k$ be a commutative Noetherian ring, and $A$ a Noetherian $k$-algebra which is projective as a $k$-module. Furthermore, let $G$ be a finite group acting on $A$. Then the following hold.

\emph{(1)} If the skew group algebra $A \rtimes G$ satisfies \emph{\textbf{Fgb}}, then so does the algebra $A$.

\emph{(2)} Suppose that the order of $G$ is an invertible element of $k$. Then $A$ satisfies \emph{\textbf{Fgb}} if and only if $A \rtimes G$ does.
\end{theorem}

The second part of the theorem has recently been proved independently in \cite{Sandoy}, for $k$ an algebraically closed field.

\begin{proof}
The skew group algebra becomes an $A$-bimodule by defining 
$$a_1 \cdot (a \otimes g) \cdot a_2 = a_1 a ({^ga_2}) \otimes g$$
This is the same as the bimodule structure we obtain by viewing $A$ as a subalgebra of $A \rtimes G$ via the injective $k$-algebra homomorphism $A \to A \rtimes G$ given by $a \mapsto a \otimes e$, where $e$ is the identity element of $G$. The decomposition
$$A \rtimes G = \bigoplus_{g \in G} A \otimes g$$
is then one of $A$-bimodules, and $A$ is isomorphic to the summand $A \otimes e$. Note also that each summand, and hence also $A \rtimes G$ itself, is projective both as a left and as a right $A$-module. Namely, as a left $A$-module, the summand $A \otimes g$ is isomorphic to $A$, and as a right $A$-module it is isomorphic to the twisted regular module $A_g$.

Denote the skew group algebra by $B$, and consider the bimodules ${_AU_B} = B$ and ${_BV_A} = B$. By the above, they are projective as one-sided modules, and $A$ is a direct summand of $U \otimes_B V$ as $A$-bimodules. Thus $A$ separably divides $B$. If the order of $G$ is an invertible element of $k$, then it follows from \cite[Theorem 1.1, part (A)]{ReitenRiedtmann} that $B$ is a direct summand of $V \otimes_A U$ as $B$-bimodules (the proof does not require $k$ to be a field, or $A$ to be an Artin algebra), hence $A$ and $B$ are separably equivalent in this case. The result now follows from Theorem \ref{thm:FGBsepequiv}.
\end{proof}

As mentioned at the beginning of this section, a theorem of Deligne provides a link between certain skew group algebras and finite symmetric tensor categories over fields of characteristic zero. We shall now use this, together with Theorem \ref{thm:skewgroupalgebras}, to establish finite generation of cohomology for such categories. Moreover, we shall also see that the so-called representation dimension of such a category is one more than the Krull dimension of its cohomology ring.

Let us now upgrade $k$ to a field (not necessarily algebraically closed), and suppose that $( \C, \ot, \unit )$ is a finite tensor category over $k$, in the sense of \cite{EGNO}. Thus $\C$ is a locally finite $k$-linear abelian category, with finitely many isomorphism classes of simple objects, all of which admit projective covers. Furthermore, there is a bifunctor $\ot \colon \C \times \C \to \C$, the tensor product, which is associative up to functorial isomorphisms, and compatible with the abelian structure of $\C$. Specifically, the tensor product is bilinear on morphisms, and satisfies the so-called pentagon axiom. Moreover, there is a two-sided unit object $\unit \in \C$ with respect to the tensor product, and this object is simple. Finally, every object of $\C$ admits both a left and a right dual object, that is, the category is rigid. By \cite[Proposition 4.2.1 and Remark 6.1.4]{EGNO}, the latter implies that the tensor product is biexact, and that $\C$ is a quasi-Frobenius category, that is, the projective objects and the injective objects coincide.

The cohomology ring of $( \C, \ot, \unit )$ is the graded $k$-algebra $\Coh^* ( \C ) = \Ext_{\C}^*( \unit, \unit )$, with the Yoneda product as multiplication. By \cite[Theorem 1.7]{Suarez-Alvarez}, this is a graded-commutative ring. If $M$ is an object of $\C$, then the tensor product $- \ot M$ induces a homomorphism
\begin{center}
\begin{tikzpicture}
\diagram{d}{3em}{3em}{
\Coh^* ( \C ) & \Ext_{\C}^* \left ( M,M \right ) \\
 };
\path[->, font = \scriptsize, auto]
(d-1-1) edge node{$\psi_M$} (d-1-2);
\end{tikzpicture}
\end{center}
of graded $k$-algebras. Thus if $M$ and $N$ are objects of $\C$, then $\Ext_{\C}^* (M,N)$ becomes a right $\Coh^* ( \C )$-module via $\psi_M$ (and the Yoneda product), and a left $\Coh^* ( \C )$-module via $\psi_N$. However, by adapting the proof of \cite[Theorem 1.1]{SnashallSolberg}, one can show that the cohomology ring acts graded-commutatively. More precisely, if $\eta \in \Coh^* ( \C )$ and $\theta \in \Ext_{\C}^* (M,N)$ are homogeneous elements, then
$$\psi_N ( \eta ) \circ \theta = (-1)^{ | \eta | | \theta | } \theta \circ \psi_M ( \eta )$$
Consequently, the cohomology ring acts on the cohomology of $\C$ basically in one way. The following was conjectured by Etingof and Ostrik in \cite{EtingofOstrik}.

\begin{conjecture}
The cohomology ring $\Coh^* ( \C )$ is Noetherian, and $\Ext_{\C}^* (M,M)$ is a finitely generated $\Coh^* ( \C )$-module for all objects $M \in \C$.
\end{conjecture}

Note that if the conjecture holds, then for all objects $M,N \in \C$, the $\Coh^* ( \C )$-module $\Ext_{\C}^* (M,N)$ is finitely generated, and not just $\Ext_{\C}^* (M,M)$ and $\Ext_{\C}^* (N,N)$. This follows from the fact that $\Ext_{\C}^* (M,N)$ is a direct summand of $\Ext_{\C}^* (M \oplus N,M \oplus N)$.

The conjecture is open in general, but it is known to be true for several important classes of tensor categories. For example, by the classical results of Evens and Venkov (cf.\ \cite{Evens, Venkov}), it holds for the category of finitely generated modules over a group algebra of a finite group. More generally, it was shown in \cite{FriedlanderSuslin} that is holds over finite dimensional cocommutative Hopf algebras. When it holds for a finite tensor category, then as shown in \cite{BerghPlavnikWitherspoon}, there is a rich theory of support varieties, just as in the classical case of group algebras and cocommutative Hopf algebras.

\begin{definition}
The finite tensor category $( \C, \ot, \unit )$ satisfies \textbf{Fgt} if the following hold: the cohomology ring $\Coh^* ( \C )$ is Noetherian, and $\Ext_{\C}^* (M,M)$ is a finitely generated $\Coh^* ( \C )$-module for all objects $M \in \C$.
\end{definition}

We have used the letter <<t>> to indicate that we are looking at tensor categories, and in order to distinguish this assumption from the finiteness conditions \textbf{Fg} and \textbf{Fgb}, which we have reserved for algebras and their Hochschild cohomology rings. However, for a finite dimensional Hopf algebra, all the three conditions make sense. Namely, given such a $k$-algebra $A$, we may ask if \textbf{Fg} (or \textbf{Fgb}) holds, that is, is it true that its Hochschild cohomology ring $\Hoch^*(A)$ is Noetherian, and that $\Ext_A^*(M,M)$ is a finitely generated $\Hoch^*(A)$-module for every $A$-module $M$ (or $\Ext_{A^e}^*(A,X)$ for every $A$-bimodule $X$)? On the other hand, the Hopf algebra structure on $A$ turns the module category $\mod A$ into a finite tensor category, with $k$ as the unit object, and then with $\Ext_A^*(k,k)$ as the cohomology ring, usually denoted just by $\Coh^* ( A )$ (it does not matter whether we use the cup product or the Yoneda product as multiplication in this ring; by \cite[Theorem 9.3.4]{Witherspoon} they are equal). Thus we may also ask if \textbf{Fgt} holds for $\mod A$: is it true that $\Coh^* ( A )$ is Noetherian, and that $\Ext_A^*(M,M)$ is a finitely generated $\Coh^* ( A )$-module for every $A$-module $M$? The following lemma, which is a strengthening of \cite[Proposition 2.9 and Proposition 3.4(1)\&(2)]{NguyenWangWitherspoon}, shows that these two finiteness conditions are equivalent. For simplicity, we shall just say that <<$A$ satisfies \textbf{Fgt}>> if the finiteness condition \textbf{Fgt} holds for the finite tensor category $\mod A$.

\begin{lemma}\label{lem:fgversusfgc}
A finite dimensional Hopf algebra satisfies \emph{\textbf{Fgt}} if and only if it satisfies \emph{\textbf{Fg}}.
\end{lemma}

\begin{proof}
Let $A$ be a finite dimensional Hopf algebra over the field $k$. If $A$ satisfies \textbf{Fgt}, then by \cite[Proposition 2.9]{NguyenWangWitherspoon} it also satisfies \textbf{Fg}. Suppose therefore that it satisfies the latter finiteness condition: the Hochschild cohomology ring $\Hoch^*(A)$ is Noetherian, and $\Ext_A^*(M,M)$ is a finitely generated $\Hoch^*(A)$-module for every $A$-module $M$.

If $M$ is an $A$-module, then since $\Ext_A^*(M,M)$ is finitely generated as a module over the Noetherian ring $\Hoch^*(A)$, it is a Noetherian ring itself. In particular, the cohomology ring $\Coh^*(A) = \Ext_A^*(k,k)$ is Noetherian. Now consider $\Ext_A^*(M,M)$ as a module over $\Coh^* ( A )$, via the ring homomorphism $- \ot_k M$; we must show that it is finitely generated as such. However, by \cite[Theorem 9.3.9]{Witherspoon}, this is the case if and only if $\Ext_A^*(k,M \ot_k D(M))$ is a finitely generated right $\Coh^*(A)$-module via the Yoneda product, where $D(M) = \Hom_k(M,k)$. The latter holds because $A$ satisfies \textbf{Fg}. Namely, $\Ext_A^*(k,M \ot_k D(M))$ is finitely generated over $\Hoch^*(A)$, via the ring homomorphism $- \ot_A k \colon \Hoch^*(A) \to \Coh^*(A)$ followed by Yoneda product. Thus $\Ext_A^*(k,M \ot_k D(M))$ must be a finitely generated right $\Coh^*(A)$-module via the Yoneda product. This shows that $A$ satisfies \textbf{Fgt}.
\end{proof}

A finite tensor category $( \C, \ot, \unit )$ is called \emph{braided} if there exist functorial isomorphisms $M \ot N \xrightarrow{b_{M,N}} N \ot M$ for all $M,N \in \C$, and these satisfy the hexagonal identities from \cite[Definition 8.1.1]{EGNO}. If moreover $b_{N,M} \circ b_{M,N}$ equals the identity on $M \ot N$ for all $M$ and $N$, then the tensor category is called \emph{symmetric}. The following result shows that every finite symmetric tensor category over a field of characteristic zero satisfies \textbf{Fgt}. The proof relies on Deligne's famous characterization from \cite{Deligne} of such categories over algebraically closed ground fields.

\begin{theorem}\label{thm:symmetricfgc}
Every finite symmetric tensor category over a field of characteristic zero satisfies \emph{\textbf{Fgt}}.
\end{theorem}

\begin{proof}
Let $k$ be a field of characteristic zero, and $( \C, \ot, \unit )$ a finite symmetric tensor category over $k$. As explained in \cite[Section 5.1]{NegronPlavnik}, for every field extension $k \subseteq K$, there is a finite tensor category $( \C_K, \ot_K, \unit_K )$ obtained as a base change of a certain weak Hopf algebroid. Since $( \C, \ot, \unit )$ is symmetric, so is $( \C_K, \ot_K, \unit_K )$, and by \cite[Lemma 5.2]{NegronPlavnik} the category $( \C, \ot, \unit )$ satisfies \textbf{Fgt} if and only if $( \C_K, \ot_K, \unit_K )$ does. Hence we may assume that the field $k$ is algebraically closed. 

By \cite[Lemma 9.11.3]{EGNO}, every finite tensor category -- in particular our category $( \C, \ot, \unit )$ -- trivially has subexponential growth. In other words, for every object $M \in \C$, there exists a natural number $n(M)$ with the property that for every integer $t \ge 0$, the length of the object $M^{\ot t}$ is at most $n(M)^t$. We may therefore apply Deligne's characterization result \cite[Th{\'e}or{\`e}me 0.6]{Deligne}: as a finite tensor category, $( \C, \ot, \unit )$ is equivalent to the category of super-representations of a certain super-group $sG$. As explained in \cite[Section 7.1]{NegronPlavnik}, using \cite[Corollary 2.3.5]{AndruskiewitschEtingofGelaki}, this implies that $( \C, \ot, \unit )$ is actually equivalent to the module category of a finite dimensional triangular Hopf $k$-algebra in the form of a skew group algebra $\wedge V \rtimes G$, where $G$ is a certain finite group, $V$ is a certain finite dimensional $k$-vector space (and a $kG$-module), and $\wedge V$ is the exterior algebra on $V$. We are therefore done if we can show that the Hopf algebra $\wedge V \rtimes G$ satisfies \textbf{Fgt}. 

By \cite[Theorem 5.5]{BerghOppermann} and \cite[Proposition 9.1]{ErdmannSolberg}, every finite dimensional exterior algebra -- in particular $\wedge V$ -- satisfies \textbf{Fg}; this can also be deduced from \cite[Theorem 4.1]{BGSS}. Now since $k$ has characteristic zero, the order of the group $G$ is trivially invertible in $k$, and hence from Lemma \ref{lem:finitenessassumptions}(2), Remark \ref{rem:conditionssatisfied}(1) and Theorem \ref{thm:skewgroupalgebras} we see that the skew group algebra $\wedge V \rtimes G$ also satisfies \textbf{Fg}. Then by Lemma \ref{lem:fgversusfgc} it also satisfies \textbf{Fgt}, and we are done.
\end{proof}

When a finite tensor category $( \C, \ot, \unit )$ satisfies \textbf{Fgt}, then the Krull dimension of its cohomology ring $\Coh^*( \C )$, defined as its rate of growth $\gamma \left ( \Coh^*( \C ) \right )$ as a graded $k$-vector space, is finite. Now define $\Ho ( \C )$ to be just $\Coh^*( \C )$ when the ground field has characteristic two, and the even part $\Coh^{2*}( \C )$ if not; thus $\Ho ( \C )$ is graded and commutative in the ordinary sense. It follows from \cite[Sections 5.3 and 5.4]{Benson2} that the Krull dimension of $\Coh^*( \C )$ equals that of $\Ho ( \C )$, and that the latter can be defined either as the rate of growth or in terms of chains of prime ideals. We shall denote this number by $\Kdim \C$, and speak of <<the Krull dimension of $\C$.>> By \cite[Theorem 4.1 and Remark 4.2]{BerghPlavnikWitherspoon}, it is equal to the complexity of the unit object $\unit$, or, equivalently, the maximal complexity obtained by the objects of $\C$. We end this section with some results showing that when \textbf{Fgt} holds, then the Krull dimension is linked to at least two important invariants of the finite tensor category in question: the representation dimension, and the dimension of the stable category.

Let $\A$ be an abelian category, and $M$ and $N$ two objects of $\A$. We then define the \emph{left $M$-resolution dimension of $N$}, denoted $\lresdim_M (N)$, to be the infimum of integers $t \ge 2$ with the property that there exists an exact sequence
\begin{center}
\begin{tikzpicture}
\diagram{d}{3em}{3em}{
0 & M_{t-2} & M_{t-3} & \cdots & M_0 & N & 0 \\
 };
\path[->, font = \scriptsize, auto]
(d-1-1) edge (d-1-2)
(d-1-2) edge (d-1-3)
(d-1-3) edge (d-1-4)
(d-1-4) edge (d-1-5)
(d-1-5) edge (d-1-6)
(d-1-6) edge (d-1-7);
\end{tikzpicture}
\end{center}
in which each $M_i$ belongs to $\add_{\A} (M)$, and which remains exact when we apply $\Hom_{\A} (M,-)$. Similarly, we define the \emph{right $M$-resolution dimension of $N$}, denoted $\rresdim_M (N)$, to be the infimum of integers $s \ge 2$ with the property that there exists an exact sequence
\begin{center}
\begin{tikzpicture}
\diagram{d}{3em}{3em}{
0 & N & M_0 & \cdots & M_{s-3} & M_{s-2} & 0 \\
 };
\path[->, font = \scriptsize, auto]
(d-1-1) edge (d-1-2)
(d-1-2) edge (d-1-3)
(d-1-3) edge (d-1-4)
(d-1-4) edge (d-1-5)
(d-1-5) edge (d-1-6)
(d-1-6) edge (d-1-7);
\end{tikzpicture}
\end{center}
in which each $M_i$ belongs to $\add_{\A} (M)$, and which remains exact when we apply $\Hom_{\A} (-,M)$.

\begin{definition}
For an abelian category $\A$, we define the \emph{representation dimension}, denoted $\repdim \A$, as follows. If $\A$ is semisimple, we set $\repdim \A =0$, and if not, we define it to be the infimum of all integers $d \ge 2$ with the property that there exists an object $M \in \A$ for which $\lresdim_M (N) \le d$ and $\rresdim_M (N) \le d$ for all objects $N \in \A$.
\end{definition}

The representation dimension was introduced for (the module categories of) Artin algebras by Auslander in \cite{Auslander1}, using a slightly different definition. He defined it to be one for semisimple algebras, and used the global dimensions of certain attached algebras if not. The motivation was to measure how far an algebra is from having finite representation type. Namely, an algebra is of finite representation type if and only if its representation dimension is at most two.

By \cite[Lemma 2.1]{EHIS}, when the abelian category $\A$ is Krull-Schmidt and has enough projective and injective objects, then 
$$\repdim \A = \inf \left \{ \gldim \Hom_{\A}(M,M) \mid M \text{ generates and cogenerates } \A \right \}$$
In particular, this is the case for finite tensor categories (cf.\ \cite[Section 1.8]{EGNO}). From the definition, we see directly that such a category has finite representation type if and only if its representation dimension is either zero (in which case the category is semisimple, i.e.\ a fusion category), or two. Therefore, as for algebras, the representation dimension of a finite tensor category should in some sense measure how far it is from having finite representation type.

\begin{remark}\label{rem:finite}
The representation dimension of a finite tensor category is always finite. Namely, as explained in \cite[Section 1.8]{EGNO}, the underlying abelian category is equivalent to the category of finitely generated modules over some finite dimensional algebra. By \cite[Corollary 1.2]{Iyama}, every such algebra has finite representation dimension.
\end{remark}

By \cite{Happel}, since a finite tensor category $( \C, \ot, \unit )$ is quasi-Frobenius, its stable category $\stable \C$ is triangulated -- in fact tensor triangulated. Recall that $\stable \C$ has the same objects as $\C$, but the morphism spaces are obtained as the  quotients of the $\C$-morphisms by the ones that factor through projective objects. The distinguished triangles correspond to the short exact sequences in $\C$. Now recall the notion of the dimension of a triangulated category $( \T, \Sigma )$ from \cite{Rouquier2}. Given subcategories $\X, \Y \subseteq \T$, we define $\X \ast \Y$ as the full subcategory of $\T$  formed by the objects $Z$ for which there exists a distinguished triangle
\begin{center}
\begin{tikzpicture}
\diagram{d}{3em}{3em}{
X & Z & Y & \Sigma X \\
 };
\path[->, font = \scriptsize, auto]
(d-1-1) edge (d-1-2)
(d-1-2) edge (d-1-3)
(d-1-3) edge (d-1-4);
\end{tikzpicture}
\end{center}
with $X \in \X$ and $Y \in \Y$. Now for an object $X \in \T$, we define $\thick_{\T}^1 (X) = \add_{\T} \left ( \{ \Sigma^t X \}_{t \in \mathbb{Z}} \right )$, and then for $n \ge 2$ 
$$\thick_{\T}^n (X) = \thick_{\T}^1 \left ( \thick_{\T}^{n-1} (X) \ast \thick_{\T}^1 (X) \right )$$
Informally, this is the full subcategory of $\T$ containing the objects that can be generated by $X$ in $n$ steps. 

\begin{definition}
The \emph{dimension} of a triangulated category $\T$, denoted $\dim \T$, is the infimum of all integers $d \ge 0$ for which there exists an object $X \in \T$ with $\thick_{\T}^{d+1} (X) = \T$.
\end{definition}

The dimension of the stable category $\stable \C$ provides a lower bound for the representation dimension. More precisely, by adapting the first part of the proof of \cite[Proposition 3.7]{Rouquier1}, we see that
$$\dim \stable \C + 2 \le \repdim \C$$
When the tensor category satisfies \textbf{Fgt}, then the following result, which follows directly from \cite[Theorem 1.1]{BIKO}, Remark \ref{rem:finite} and the above inequality, links these two invariants to the Krull dimension, and shows that all these invariants are finite.

\begin{theorem}\label{thm:bounds}
If $( \C, \ot, \unit )$ is a finite tensor category satisfying \emph{\textbf{Fgt}}, then
$$\Kdim \C \le \dim \stable \C + 1 \le \repdim \C -1 < \infty$$
\end{theorem}

Finally, we combine this result with Theorem \ref{thm:symmetricfgc}, to obtain bounds for all finite symmetric tensor categories over fields of characteristic zero. In fact, when the ground field is algebraically closed, then we obtain exact values, once again thanks to Deligne's characterization.

\begin{theorem}\label{thm:symmetricrepdim}
If $( \C, \ot, \unit )$ is a finite symmetric tensor category over a field of characteristic zero, then 
$$\Kdim \C \le \dim \stable \C + 1 \le \repdim \C -1 < \infty$$
Moreover, when the field is algebraically closed, then equalities hold among the finite integers:
$$\Kdim \C = \dim \stable \C + 1 = \repdim \C -1 < \infty$$
\end{theorem}

\begin{proof}
The first part is the result of combining Theorem \ref{thm:symmetricfgc} and Theorem \ref{thm:bounds}. For the second part, suppose that the ground field $k$ is algebraically closed. Recall from the proof of Theorem \ref{thm:symmetricfgc} that our tensor category is equivalent to the module category of a finite dimensional triangular Hopf $k$-algebra in the form of a skew group algebra $\wedge V \rtimes G$, where $G$ is a certain finite group, $V$ is a certain finite dimensional $k$-vector space (and a $kG$-module), and $\wedge V$ is the exterior algebra on $V$. Recall also that $\wedge V$ and $\wedge V \rtimes G$ are separably equivalent, and that they satisfy \textbf{Fg}.

As pointed out after the proof of Theorem \ref{thm:symmetricfgc}, the Krull dimension of $\C$ equals the maximal complexity obtained by its objects. This equals the maximal complexity obtained by the finitely generated left $( \wedge V \rtimes G )$-modules, and by \cite[Theorem 1]{Peacock} the latter integer equals the maximal complexity obtained over $\wedge V$. However, the exterior algebra is local, with $k$ as its simple module. The maximal complexity therefore equals that of $k$, and it is well known that this is the same as the dimension of $V$ as a $k$-vector space. This again equals $\ell \ell \left ( \wedge V \right ) - 1$, where $\ell \ell (A)$ denotes the Loewy length of a finite dimensional algebra $A$. By \cite[Theorem 1.1 and Theorem 1.3(e)(ii)]{ReitenRiedtmann}, the Loewy lengths of $\wedge V$ and $\wedge V \rtimes G$ are the same, hence 
$$\ell \ell ( \wedge V \rtimes G ) - 1 = \Kdim \C$$
Finally, by \cite[Theorem 1.1 and Theorem 1.3(c)(iii)]{ReitenRiedtmann}, the algebra $\wedge V \rtimes G$ is selfinjective, since the exterior algebra $\wedge V$ is. It then follows from \cite[Section III.5]{Auslander1} that
$$\repdim \C = \repdim ( \wedge V \rtimes G ) \le \ell \ell ( \wedge V \rtimes G )$$
and so by combining what we have shown we obtain
$$\ell \ell ( \wedge V \rtimes G ) - 1 = \Kdim \C \le \dim \stable \C + 1 \le \repdim \C -1 \le \ell \ell ( \wedge V \rtimes G ) -1$$
This concludes the proof.
\end{proof}

\section{Some further applications}\label{sec:applications2}

In this final section, we explore some further applications of Theorem \ref{thm:FGBsepequiv}. As before, let $k$ be a commutative Noetherian ring, and $A$ and $B$ two Noetherian $k$-algebras which are projective as $k$-modules. Again, unless otherwise specified, all modules are assumed to be finitely generated left modules. We state all the results in this section in terms of the finiteness assumption \textbf{Fgb}; by combining the results with Lemma \ref{lem:finitenessassumptions} and Remark \ref{rem:conditionssatisfied}, one obtains versions with the assumption \textbf{Fg}, for example when $k$ is an algebraically closed or a perfect field.

To any algebra, we can associate a number of triangulated categories using modules. When the corresponding triangulated categories are equivalent for two algebras, one may ask if \textbf{Fgb} is an invariance under the particular equivalence in question. For example, if $A$ and $B$ are derived equivalent, meaning that their bounded derived categories $\derived ( \mod A )$ and $\derived ( \mod B )$ are equivalent as triangulated categories, then is it true that $A$ satisfies \textbf{Fgb} if and only if $B$ does? This is indeed true, and it follows more or less immediately from Rickard's famous description of derived equivalences (see also \cite{KulshammerPsaroudakisSkartsaterhagen}). Namely, by \cite[Theorem 6.4(b)\&(e)]{Rickard1}, any equivalence $\derived ( \mod A ) \to \derived ( \mod B )$ is induced by a tilting complex. Then since both $A$ and $B$ are projective as $k$-modules, it follows from \cite[Corollary 2.3 and Proposition 2.5]{Rickard2} that $A$ satisfies \textbf{Fgb} if and only if $B$ does. Here we are using the fact that we can determine finite generation of cohomology by looking at the stalk complexes, since every object can be filtered by such, using finitely many distinguished triangles. Thus, for example, $\Ext_{A^e}^*(A,X)$ is a Noetherian right $\Hoch^*(A)$-module for every $A$-bimodule $X$ if and only if $\Ext_{A^e}^*(A,X^*)$ is a Noetherian right $\Hoch^*(A)$-module for every object $X^* \in \derived ( \mod A )$.

What about singular equivalences? Recall that the singularity category $\sing (A)$ of $A$ is the Verdier quotient $\derived ( \mod A ) / \perf A$, where $\perf A$ is the thick subcategory of $\derived ( \mod A )$ formed by the perfect complexes. It is not known whether \textbf{Fgb} is invariant under singular equivalences. That is, if the singularity categories $\sing (A)$ and $\sing (B)$ are equivalent as triangulated categories, and $A$ satisfies \textbf{Fgb}, then it is unknown whether $B$ must also satisfy \textbf{Fgb}. However, it is known to hold for some special kinds of singular equivalences, and as we shall see, these results can be directly recovered from Theorem \ref{thm:FGBsepequiv}.

Following \cite{Broue}, the algebras $A$ and $B$ are called \emph{stably equivalent of Morita type} if there exist bimodules ${_AU_B}$ and ${_BV_A}$ -- projective on both sides -- with the property that $U \ot_B V \simeq A \oplus P$ and $V \ot_A U \simeq B \oplus Q$ for some projective bimodules ${_AP_A}$ and ${_BQ_B}$. If we relax the requirements on $P$ and $B$, and just assume that they have finite projective dimension as bimodules, then $A$ and $B$ are called \emph{singularly equivalent of Morita type} (this latter concept seems to originate from the unpublished manuscript \cite{ChenSun}). When this holds, then the tensor product $V \ot_A -$ induces a singular equivalence $\sing (A) \to \sing (B)$ (cf.\ \cite[Proposition 2.3]{ZhouZimmermann}). In particular, when $A$ and $B$ are selfinjective, then a stable equivalence of Morita type induces an equivalence between the stable module categories. Since the equivalences just defined are special kinds of separable equivalences, we can use Theorem \ref{thm:FGBsepequiv} directly and obtain the following result.

\begin{theorem}\label{thm:stableequiv}
Let $k$ be a commutative Noetherian ring and $A$ and $B$ two Noetherian $k$-algebras which are projective as $k$-modules. Furthermore, suppose that $A$ and $B$ are singularly equivalent of Morita type (as happens for example when they are stably equivalent of Morita type). Then $A$ satisfies \emph{\textbf{Fgb}} if and only if $B$ does.
\end{theorem}

We shall now generalize this result, and also recover the main result from \cite{Skartsaterhagen}. We start by introducing a generalized notion of separable equivalence for algebras. Recall first that a \emph{syzygy} of a module over a ring is a kernel in some projective resolution of the module.

\begin{definition}
The algebra $B$ \emph{separably quasi-divides} the algebra $A$ if the following hold: there exist bimodules ${_AU_B}$ and ${_BV_A}$ -- projective on both sides -- with the property that some $B$-bimodule syzygy of $B$ is a direct summand of $V \ot_A U$. If there exists such a pair of bimodules such that in addition some $A$-bimodule syzygy of $A$ is a direct summand of $U \ot_B V$, then $A$ and $B$ are \emph{separably quasi-equivalent}.
\end{definition}

Thus $B$ separably quasi-divides $A$ if there exists an exact sequence
\begin{center}
\begin{tikzpicture}
\diagram{d}{3em}{3em}{
0 & K & P_{n-1} & \cdots & P_0 & B & 0 \\
 };
\path[->, font = \scriptsize, auto]
(d-1-1) edge (d-1-2)
(d-1-2) edge (d-1-3)
(d-1-3) edge (d-1-4)
(d-1-4) edge (d-1-5)
(d-1-5) edge (d-1-6)
(d-1-6) edge (d-1-7);
\end{tikzpicture}
\end{center}
of $B$-bimodules, with each $P_i$ projective, and such that $K$ is a direct summand of $V \ot_A U$. The $K$ is then a syzygy of $B$ of \emph{degree} $n$, or an $n$th syzygy of $B$.

\begin{remark}\label{rem:syzygies}
(1) Since $B$ is a syzygy of itself, of degree zero, the notion of separable quasi-division/quasi-equivalence generalizes that of separable division/equivalence.

(2) Suppose that $A$ and $B$ are Artin $k$-algebras (for example when $k$ is a field). Then so are the enveloping algebras $A^e$ and $B^e$. Over such algebras, every (finitely generated left) module admits a minimal projective resolution, which is unique up to isomorphism. This resolution is a summand in every other projective resolution of the module. We denote by $\Omega_{A^e}^n(A)$ (and similarly for $B$) the $n$th kernel in the minimal projective $A$-bimodule resolution of $A$; this bimodule is then a summand of every $n$th syzygy of $A$. Now if $B$ separably quasi-divides $A$, with $U,V$ and $K$ as above, then $\Omega_{B^e}^n(B)$ is a direct summand of $K$, and therefore also of $V \ot_A U$. Thus $B$ separably quasi-divides $A$ if and only if there exist bimodules ${_AU_B}$ and ${_BV_A}$ -- projective on both sides -- with the property that $\Omega_{B^e}^n(B)$ is a direct summand of $V \ot_A U$ for some $n \ge 0$. Similarly, $A$ and $B$ are separably quasi-equivalent if and only if there exist such bimodules $U$ and $V$ for which $\Omega_{A^e}^m(A)$ is a direct summand of $U \ot_B V$ for some $m \ge 0$, and $\Omega_{B^e}^n(B)$ is a direct summand of $V \ot_A U$ for some $n \ge 0$.

(3) In the definition of separable quasi-equivalence, we do not require that the syzygies of $A$ and $B$ involved are of the same degree. That is, there is supposed to exist an $m$th syzygy of $A$ as a summand of $U \ot_B V$, and an $n$th syzygy of $B$ as a summand of $V \ot_A U$, but we do not require that $m=n$. Thus the notion of separable quasi-equivalence generalizes that of \emph{singular equivalence of Morita type with level}, introduced in \cite{Wang}.
\end{remark}
 
Recall that a $k$-algebra is \emph{Gorenstein} if its injective dimensions as a left and as a right module over itself are both finite; by \cite[Lemma A]{Zaks}, these numbers are then the same. The following result shows that the condition \textbf{Fgb} is an invariant under separable quasi-equivalence, provided the enveloping algebras are Gorenstein.

\begin{theorem}\label{thm:gorenstein}
Let $k$ be a commutative Noetherian ring and $A$ and $B$ two Noetherian $k$-algebras which are projective as $k$-modules. Suppose that $A$ satisfies \emph{\textbf{Fgb}}, that $B$ separably quasi-divides $A$, and that the enveloping algebra $B^e$ is Gorenstein. Then $B$ also satisfies \emph{\textbf{Fgb}}. In particular, if $A$ and $B$ are separably quasi-equivalent, and $A^e$ and $B^e$ are Gorenstein, then $A$ satisfies \emph{\textbf{Fgb}} if and only if $B$ does.
\end{theorem}

\begin{proof}
By definition, there exist bimodules ${_AU_B}$ and ${_BV_A}$ -- projective on both sides -- with the property that some $B$-bimodule syzygy $K$ of $B$ is a direct summand of $V \ot_A U$. The main part of the proof of Theorem \ref{thm:FGBsepequiv} only uses the fact that $U$ and $V$ are projective as one-sided modules, to conclude that $\Ext_{B^e}^* \left ( V \otimes_A U, Y \right )$ is a Noetherian right module over $\Ext_{B^e}^* \left ( V \otimes_A U, V \otimes_A U \right )$, for every $B$-bimodule $Y$. Now since $K$ is a direct summand of $V \otimes_A U$, it follows from \cite[Lemma 4.3]{Linckelmann} that $\Ext_{B^e}^* \left ( K, Y \right )$ is a Noetherian right module over $\Ext_{B^e}^* \left ( K,K \right )$, for every $B$-bimodule $Y$.

Since $K$ is a syzygy of $B$, there exists a projective $B$-bimodule resolution
\begin{center}
\begin{tikzpicture}
\diagram{d}{3em}{3em}{
\cdots & P_2 & P_1 & P_0 & B & 0 \\
 };
\path[->, font = \scriptsize, auto]
(d-1-1) edge (d-1-2)
(d-1-2) edge node{$d_2$} (d-1-3)
(d-1-3) edge node{$d_1$} (d-1-4)
(d-1-4) edge node{$d_0$} (d-1-5)
(d-1-5) edge (d-1-6);
\end{tikzpicture}
\end{center}
of $B$ in which $K$ is a kernel; say $K = \Ker d_s$, or $K = B$. Now take a homogeneous element $\eta \in \Hoch^* (B)$ of degree $t \ge s+1$, represented by a map $\Ker d_{t-1} \to B$. Lifting this map along the projective resolution, we obtain a map $\Ker d_{t+s} \to K$, representing an element in $\Ext_{B^e}^t \left ( K,K \right )$. The resulting map $\Hoch^t (B) \to \Ext_{B^e}^t \left ( K,K \right )$ respects the Yoneda product, and since $B^e$ is Gorenstein, it is an isomorphism of $k$-modules for $t \gg 0$. Then for every $B$-bimodule $Y$, since $\Ext_{B^e}^* \left ( K, Y \right )$ is a Noetherian right module over $\Ext_{B^e}^* \left ( K,K \right )$, we see that $\oplus_{n=n_Y}^{\infty} \Ext_{B^e}^n \left ( B, Y \right )$ is a Noetherian right module over $\Hoch^* (B)$ for some $n_Y$, depending on $Y$. But $\Ext_{B^e}^n \left ( B, Y \right )$ is a finitely generated $k$-module for all $n$, hence $\Ext_{B^e}^* \left ( B, Y \right )$ is a Noetherian right module over $\Hoch^* (B)$. This shows that $B$ satisfies \textbf{Fgb}, by Remark \ref{rem:conditionssatisfied}(2).
\end{proof}

\begin{remark}\label{rem:gorenstein}
(1) By \cite[Lemma 2.1]{BerghJorgensen}, the enveloping algebra of a finite dimensional Gorenstein algebra over a field is again Gorenstein. Hence when $k$ is a field, then in Theorem \ref{thm:gorenstein} we can assume that $A$ and $B$ are Gorenstein.

(2) Suppose that $k$ is a field and that $A / \mathfrak{r}_A \otimes_k A / \mathfrak{r}_A$ and $B / \mathfrak{r}_B \otimes_k B / \mathfrak{r}_B$ are semisimple, where $\mathfrak{r}_A$ is the radical of $A$, and $\mathfrak{r}_B$ is the radical of $B$; this happens for example when $k$ is algebraically closed or perfect, cf.\ Remark \ref{rem:conditionssatisfied}. Then by combining the above remark with Lemma \ref{lem:finitenessassumptions} and Theorem \ref{thm:gorenstein}, we recover \cite[Theorem 7.4]{Skartsaterhagen}.
\end{remark}

We end this paper with an application of Theorem \ref{thm:FGBsepequiv} to generalized trivial extension algebras. Let $X$ be an $A$-bimodule, and consider the $k$-algebra $A \ltimes X$. As a $k$-module, this is the direct sum $A \oplus X$, with multiplication
$$(a,x)(b,y) = (ab, ay+xb)$$
When $X = D(A) = \Hom_k(A,k)$, this is the \emph{trivial extension} of $A$. The following result shows that when the bimodule $X$ is projective as a one-sided module, then $A$ satisfies \textbf{Fgb} whenever $A \ltimes X$ does.

\begin{theorem}\label{thm:trivialextension}
Let $k$ be a commutative Noetherian ring and $A$ a Noetherian $k$-algebra which is projective as a $k$-module. Furthermore, let ${_AX_A}$ be a bimodule which is projective as a left and as a right $A$-module. Then if the algebra $A \ltimes X$ satisfies \emph{\textbf{Fgb}}, so does $A$.
\end{theorem}

\begin{proof}
The natural $k$-algebra inclusion $A \to A \ltimes X$ turns $A \ltimes X$ into an $A$-bimodule, and as such it is equal to $A \oplus X$. Now let $U$ and $V$ be $A \ltimes X$, considered as an $A$-$B$-bimodule and a $B$-$A$-bimodule, respectively. These are both projective as one-sided modules, and $U \otimes_B V$ is isomorphic to $A \ltimes X$ as an $A$-bimodule. Since $A$ is a direct summand of $A \ltimes X$ as an $A$-bimodule, we see that $A$ separably divides $A \ltimes X$. The result now follows from Theorem \ref{thm:FGBsepequiv}.
\end{proof}

We can apply this result to certain path algebras. Suppose that $k$ is a field, $Q$ a finite quiver, and $A$ an algebra of the form $kQ/ \az$ for some admissible ideal $\az \subseteq kQ$. Given an arrow $\alpha$ in $Q$, we may form the \emph{arrow removal algebra} $A / ( \alpha )$.

\begin{corollary}\label{cor:arrowremoval}
Let $k$ be a field and $A$ a $k$-algebra of the form $kQ / \az$, where $Q$ is a finite quiver and $\az$ an admissible ideal in $kQ$. Furthermore, let $\alpha$ be an arrow in $Q$ not belonging to any minimal generating set of $\az$. Then if the algebra $A$ satisfies \emph{\textbf{Fgb}}, so does the arrow removal algebra $A / ( \alpha )$.
\end{corollary}

\begin{proof}
Denote the arrow removal algebra $A / ( \alpha )$ by $B$, and the starting and ending vertices of $\alpha$ by $v_s$ and $v_t$, respectively. By \cite[Theorem A]{GreenPsaroudakisSolberg}, the algebra $A$ is isomorphic to $B \ltimes X$, where $X$ is the $B$-bimodule $Be_s \otimes_k e_tB$, and $e_s$ and $e_t$ are the trivial paths corresponding to $v_s$ and $v_t$, respectively. Since $X$ is projective as a left and as a right $B$-module, the result follows from Theorem \ref{thm:trivialextension}.
\end{proof}

\begin{remark}\label{rem:arrowremoval}
(1) Note that in the corollary, both $A$ and the arrow removal algebra $A / ( \alpha )$ are quotients of path algebras by certain admissible ideals. Therefore, from Lemma \ref{lem:finitenessassumptions} and Remark \ref{rem:conditionssatisfied}, we see that the result can be stated in terms of the finiteness condition \textbf{Fg}.

(2) As shown in \cite[Main Theorem]{ErdmannPsaroudakisSolberg}, the converse of the corollary actually holds as well. In other words, $A$ satisfies \textbf{Fg} if and only if the arrow removal algebra $A / ( \alpha )$ does. This is proved by using the fact that the category of $A$-modules is a cleft extension of the category of modules over $A / ( \alpha )$.
\end{remark}


\end{document}